\documentclass[12pt]{article}

\usepackage{amsmath}
\usepackage{amsthm}
\usepackage{graphicx}
\usepackage{amsfonts}
\usepackage{fullpage}
\usepackage{ amssymb }
\usepackage{cite}
\usepackage{hyperref}
\usepackage{comment}
\newcommand{\num}{}
\newcommand{\diff}{\textrm{d}}

\title{Robin's inequality for $\theight$-free integers}
\author{Thomas Morrill\footnote{Supported by Australian Research Council Discovery Project DP160100932 }\\ School of Science\\ The University of New South Wales Canberra, Australia\\ \href{mailto:t.morrill@adfa.edu.au}{\nolinkurl{t.morrill@adfa.edu.au}}\\
\\
David John Platt\footnote{Supported by Australian RC Discovery Project DP160100932 and  EPSRC Grant EP/K034383/1.}\\School of Mathematics\\ University of Bristol, Bristol, UK\\ \href{dave.platt@bris.ac.uk}{\nolinkurl{dave.platt@bris.ac.uk}}}
\date{\today}

\newcommand{\RIheight}{13.114\,85}
\newcommand{\theight}{20}
\newcommand{\nextheight}{21}

\newtheorem{theorem}{Theorem}
\newtheorem{cor}[theorem]{Corollary}
\newtheorem{lemma}[theorem]{Lemma}

\begin{document}
\maketitle

\begin{abstract}
	\noindent
	In 1984, Robin showed that the Riemann Hypothesis for $\zeta$ is equivalent to demonstrating $\sigma(n) < e^\gamma n \log \log n$ for all $n > 5040$.
	Robin's inequality has since been proven for various infinite families of power-free integers: $5$-free integers, $7$-free integers, and $11$-free integers. We extend these results to cover $\theight$-free integers.
\end{abstract}

In 1984, Robin gave an equivalent statement of the Riemann Hypothesis for $\zeta$ involving the divisors of integers.
\begin{theorem}[Robin \cite{RobinIneq}]
	The Riemann Hypothesis is true if and only if for all $n > \num{5040}$,
	\begin{align} \label{RI}
		\sigma(n) < e^\gamma n \log \log n, \tag{RI}
	\end{align}
	where $\sigma(n)$ is the sum of divisors function and $\gamma$ is the Euler--Mascheroni constant.
\end{theorem}
\noindent
Since then, \eqref{RI} has become known as Robin's inequality.
There are twenty-six known counterexamples to \eqref{RI}, of which $\num{5040}$ is the largest \cite{RobinThm}.

Robin's inequality has been proven for various infinite families of integers, in particular the $t$-free integers.
Recall that $n$ is called \emph{$t$-free} if $n$ is not divisible by the $t$th power of any prime number, and \emph{$t$-full} otherwise.
In 2007, Choie, Lichiardopol, Moree, and Sol\'e \cite{5-free } showed that \eqref{RI} holds for all $5$-free integers greater than $\num{5040}$.
Then, in 2012, Planat and Sol\'e \cite{7-free} improved this result to \eqref{RI} for $7$-free integers greater than $\num{5040}$, which was followed by Broughan and Trudgian \cite{11-free} with \eqref{RI} for $11$-free integers greater than $\num{5040}$ in 2015.
By updating Broughan and Trudgian's work, we prove our main theorem.

\begin{theorem}\label{analytics}
	Robin's inequality holds for $\theight$-free integers greater than $\num{5040}$.
\end{theorem}
\noindent
Since there are no $\theight$-full integers less than $\num{5041}$, we may give a cleaner statement for Robin's theorem.
\begin{cor}
	The Riemann Hypothesis is true if and only if \eqref{RI} holds for all $\theight$-full integers.
\end{cor}

\section{A bound for $t$-free integers}

Sol\'e and Planat \cite{7-free} introduced the generalized Dedekind $\Psi$ function
\begin{align*}
	\Psi_t(n) := n \prod_{p | n}( 1 + p^{-1} + \dots + p^{-(t-1)}) = n \prod_{p | n}\frac{1-p^{-t}}{1 - p^{-1}}.
\end{align*}
Since
\begin{align*}
	\sigma(n) = n \prod_{p^a || n} (1 + p^{-1} + \dots + p^{-a}),
\end{align*}
we see that $\sigma(n) \leq \Psi_t(n)$, provided that $n$ is $t$-free.
Thus, we study the function
\begin{align*}
	R_t(n) := \frac{\Psi_t(n)}{n \log \log n}.
\end{align*}
\noindent
By Proposition 2 of \cite{7-free}, it is sufficient to consider $R_t$ only at the primorial numbers $p_n\# = \prod_{k=1}^n p_k$ where $p_k$ is the $k$th prime. Compare this to the role of colossally abundant numbers in \eqref{RI} by Robin \cite{RobinIneq}.

Using equation $(2)$ of Broughan and Trudgian \cite{11-free}, we have for $n \geq 2$
\begin{align*}
	R_t(p_n\#)
	= \frac{p_n\# \prod_{p \leq p_n} \frac{1-p^{-t}}{1 - p^{-1}}}{p_n\# \log \log p_n\#}
	= \frac{ \prod_{p > p_n} (1-p^{-t})^{-1}}{\zeta(t) \log \vartheta(p_n)} \prod_{p \leq p_n} (1 - p^{-1})^{-1}
\end{align*}
where $\vartheta(x)$ is the Chebyshev function $\sum_{p\leq x}\log p$.

In Sections \ref{bound1} and \ref{bound}, we construct two non-increasing functions, $g_B(w;t)$ and $g_\infty(w;t)$ such that for some constants $x_0$, $B$ we have for $x_0\leq p_n \leq B$
$$
g_B(p_n;t)\geq R_t(p_n\#)\exp(-\gamma)
$$
and for $p_n>B$
$$
g_\infty(p_n;t)\geq R_t(p_n\#)\exp(-\gamma).
$$
\noindent
For a given $t\geq 2$, if we can show that all $t$-free numbers $5\,040<n\leq p_k\#$ satisfy \eqref{RI}, that $g_B(p_k;t)<1$ and that $g_\infty(B;t)<1$, then we are done.
  
\section{Deriving $g_B(p_n;t)$} \label{bound1}

We start with some lemmas.

\begin{lemma}\label{partial_RH}
Let $\rho$ be a non-trivial zero of the Riemann zeta function with positive imaginary part $\leq 3\cdot 10^{12}$. Then $\Re \rho=1/2$.
\end{lemma}
\begin{proof}
See Theorem $1$ of \cite{Platt-RH}.
\end{proof}

\begin{lemma}\label{lem:theta}
Let $B=2.169\cdot 10^{25}$. Then we have
$$\left|\vartheta(x)-x\right|\leq\frac{1}{8\pi}\sqrt{x}\log^2 x\quad\textrm{for }599\leq x\leq B.
$$
\end{lemma}
\begin{proof}
Given that one knows Riemann Hypothesis to height $T$, \cite{Buthe-RH} tells us that we may use Schoenfeld's bounds from \cite{Schoenfeld1976} but restricted to $B$ such that
$$
4.92\sqrt{\frac{B}{\log B}}\leq T.
$$
Using $T=3\cdot 10^{12}$ from Lemma \ref{partial_RH} we find $B=2.169\cdot 10^{25}$ is admissible.
\end{proof}

\begin{lemma}\label{lem:theta55}
Let $\log x \geq 55$. Then
$$
|\vartheta(x)-x|\leq 1.388\cdot 10^{-10}x +1.4262\sqrt{x}
$$
or
$$
|\vartheta(x)-x|\leq 1.405\cdot 10^{-10}x.
$$
\end{lemma}
\begin{proof}
From Table $1$ of \cite{Dusart-estimates} we have for $x>\exp(55)$
$$
|\psi(x)-x| \leq 1.388\cdot 10^{-10} x
$$
so that by Theorem $13$ of \cite{Rosser62} we get, again for $x>\exp(55)$, that
$$
|\vartheta(x)-x|\leq 1.388\cdot 10^{-10} x +1.4262\sqrt{x}.
$$
The second bound follows trivially.
\end{proof}

\begin{lemma}\label{lem:C1}
Take $B$ as above and define
$$
C_1=\int\limits_B^\infty \frac{(\vartheta(t)-t)(1+\log t)}{t^2\log^2 t} \diff t.
$$
Then $C_1\leq 2.645\cdot 10^{-9}$.
\end{lemma}
\begin{proof}
We split the integral at $X_0=\exp(2000)$, apply Lemma \ref{lem:theta55} and consider
$$
1.405\cdot 10^{-10}\int\limits_B^{X_0}\frac{1+\log t}{t\log^2 t}\diff t\leq 1.430\cdot 10^{-10}\int\limits_B^{X_0}\frac{\diff t}{t\log t}\leq 5.055\cdot 10^{-10}.
$$
For the tail of the integral, we use 
$$
|\vartheta (x)-x|\leq 30.3 x\log^{1.52} x \exp(-0.8\sqrt{\log x})
$$
from Corollary $1$ of \cite{Platt-Pintz}, valid for $x \geq X_0$. We can then majorise the tail with
$$
30.3 \int\limits_{X_0}^\infty \frac{\log t\exp(-0.8\sqrt{\log t})}{t}\diff t
$$
which is less than $2.139\cdot 10^{-9}$.

\end{proof}

\begin{lemma}
Take $B$, $C_1$ as above and let $599\leq x \leq B$. Then

$$
\prod\limits_{p\leq x}\left(1-\frac{1}{p}\right)\geq \frac{\exp(-\gamma)}{\log x}\exp\left(\frac{1.02}{(x-1)\log x}+\frac{\log x}{8\pi\sqrt{x}}+C_1+\frac{(\log x+3)\sqrt{B}-(\log B+3)\sqrt{x}}{4\pi\sqrt{xB}}\right).
$$
\end{lemma}
\begin{proof}
Let $M$ be the Meissel-Mertens constant
$$
M=\gamma+\sum\limits_p(\log(1-1/p)+1/p).
$$
Then by $4.20$ of \cite{Rosser62} we have
$$
\left|\sum\limits_{p\leq x}\frac{1}{p}-\log\log x -M\right|\leq\frac{|\vartheta(x)-x|}{x\log x}+\int\limits_x^\infty\frac{|\vartheta(t)-t|(1+\log t)}{t^2\log^2 t}\diff t.
$$
Since $599\leq x\leq B$ we can use Lemma \ref{lem:theta} to bound the first term with
$$
\frac{\log x}{8\pi\sqrt{x}}.
$$
We can split the integral at $B$ and over the range $[B,\infty)$ use the bound from Lemma \ref{lem:C1}. This leaves the range $[x,B]$ where we can use Lemma \ref{lem:theta} and a straightforward integration yields a contribution of
$$
\frac{(\log x+3)\sqrt{B}-(\log B+3)\sqrt{x}}{4\pi\sqrt{xB}}.
$$
We then simply follow the method used to prove Theorem $5.9$ of \cite{Dusart-estimates} with our bounds in place of
$$
\frac{\eta_k}{k\log^k x}+\frac{(k+2)\eta_k}{(k+1)\log^{k+1} x}.
$$ 
\end{proof}

We also need Lemma 2 of \cite{7-free}.
\begin{lemma}[Sol\'e and Planat \cite{7-free}]
	For $n \geq 2$,
	\begin{align*}
		\prod_{p > p_n} \frac{1}{1-p^{-t}} \leq \exp(2/p_n).
	\end{align*}
\end{lemma}

Putting all this together, we have the following.
\begin{lemma}
Define
$$
g_B(p_n;t)=\frac{\exp\left(\frac{2}{p_n}+\frac{1.02}{(p_n-1)\log p_n}+\frac{\log p_n}{8\pi\sqrt{p_n}}+C_1+\frac{(\log p_n+3)\sqrt{B}-(\log B+3)\sqrt{p_n}}{4\pi\sqrt{p_nB}}\right)\log p_n}{\zeta(t)\log\left(p_n-\frac{\sqrt{p_n}\log^2 p_n}{8\pi}\right)}.
$$
Then for $t\geq 2$ and $599\leq p_n\leq B=2.169\cdot 10^{25}$ we have $g_B(p_n;t)$ non-increasing in $n$ and $R_t(p_n\#)\leq \exp(\gamma)g_B(p_n;t)$.
\end{lemma}

\section{Deriving $g_\infty(p_n;t)$} \label{bound}

We will need a further bound.

\begin{theorem}
	For $x \geq 767\,135\,587$,
	\begin{align*}
		\prod_{p \leq x} \frac{p}{p-1}
		\leq e^\gamma \log x \exp\left(\frac{1.02}{(x-1)\log x}+\frac{1}{6\log^3 x}+\frac{5}{8\log^4 x}\right).
	\end{align*}
\end{theorem}
\begin{proof}
This is the last display on page $245$ of \cite{Dusart-estimates} with $k=3$ so that $\eta_k=0.5$. 
\end{proof}

We can now deduce
\begin{theorem}
Define
$$
g_\infty(p_n;t)=\frac{\exp\left(\frac{2}{p_n}+\frac{1.02}{(p_n-1)\log p_n}+\frac{1}{6\log^3 p_n}+\frac{5}{8\log^4 p_n}\right)\log p_n}{\zeta(t)\log\left(p_n-1.338\cdot 10^{-10}p_n-1.4262\sqrt{p_n}\right)}.
$$
Then for $t \geq 2$ and $\log p_n \geq 55$ we have
\begin{align*}
	R_t(p_n\#) \leq e^\gamma g_\infty(p_n;t)
\end{align*}
and $g_\infty(p_n;t)$ is non-increasing in $n$.
\end{theorem}

\section{Computations}

The proof rests on Briggs' work \cite{10^10^10} on the colossally abundant numbers, which implies \eqref{RI} for $\num{5040} < n \leq 10^{(10^{10})}$. We extend this result with the following Theorem:
\begin{theorem}{\label{Platt}}
	Robin's inequality holds for all $\num{5040} < n \leq 10^{(10^{\RIheight})}$.
\end{theorem}
\begin{proof}
We implemented Brigg's algorithm from \cite{10^10^10} but using extended precision ($100$ bits) and interval arithmetic to carefully manage rounding errors. The final $n$ checked was
\begin{align*}
29\,996\,208\,012\,611\#&\cdot 7\,662\,961\#\cdot 44\,293\#\cdot 3\,271\#\cdot 666\#\cdot 233\#\cdot 109\#\cdot 61\#\\
&\cdot 37\#\cdot 23\#\cdot 19\#\cdot (13\#)^2\cdot (7\#)^4 \cdot (5\#)^3\cdot (3\#)^{10}\cdot 2^{19}.
\end{align*}
\end{proof}

\begin{cor}
Robin's inequality holds for all $13\#\leq n \leq 29\,996\,208\,012\,611\#$.
\end{cor}

We are now in a position to prove Theorem \ref{analytics}. We find that
$$
g_B(29\,996\,208\,012\,611;\theight)< 1
$$
and
$$
g_\infty(B;\theight)< 1
$$
and the result follows.

\section{Comments}

In terms of going further with this method, we observe that both
$$
g_B(29\,996\,208\,012\,611;\nextheight)> 1
$$
and
$$
g_\infty(B;\nextheight)>1
$$
so one would need improvements in both. We only pause to note that one of the inputs to Dusart's unconditional bounds that feed into $g_\infty$ is again the height to which the Riemann Hypothesis is known\footnote{Dusart uses $T\geq 2\,445\,999\,556\,030$.}, so the improvements from Lemma \ref{partial_RH} could be incorporated.

Finally, we observe that if $R_t(p_n\#)$ could be shown to be decreasing in $n$, then our lives would have been much easier.

\section{Acknowledgements}

The authors would like to thank Pierre Dusart and Keith Briggs for helpful conversations and Keith Briggs for sharing his code.

\end{document}